\documentclass[11pt]{amsart}
\usepackage{mathrsfs,latexsym,amsfonts,amssymb}
\newtheorem{theorem}{Theorem}[section]
\newtheorem{lemma}[theorem]{Lemma}
\newtheorem{corollary}[theorem]{Corollary}
\newtheorem{question}[theorem]{Question}
\newtheorem{remark}[theorem]{Remark}
\theoremstyle{definition}
\newtheorem{definition}[theorem]{Definition}

\newtheorem{problem}[theorem]{Problem}

\begin{document}

\title[Hyperspaces with a countable character of closed subsets]
{Hyperspaces with a countable character of closed subsets}

\author{Chuan Liu}
\address{(Chuan Liu)Department of Mathematics,
Ohio University Zanesville Campus, Zanesville, OH 43701, USA}
\email{liuc1@ohio.edu}

 \author{Fucai Lin}
 \address{(Fucai Lin) 1. School of mathematics and statistics, Minnan Normal University, Zhangzhou 363000, P. R. China; 2. Fujian Key Laboratory of Granular Computing and Application, Minnan Normal University, Zhangzhou 363000, P. R. China}
 \email{linfucai@mnnu.edu.cn}

\thanks{The second author is supported by the Key Program of the Natural Science Foundation of Fujian Province (No: 2020J02043), the NSFC (No. 11571158), the Institute of Meteorological Big Data-Digital Fujian and Fujian Key Laboratory of Data Science and Statistics.}


\begin{abstract}
For a regular space $X$, the hyperspace $(CL(X), \tau_{F})$ (resp.,
$(CL(X), \tau_{V})$) is the space of all nonempty closed subsets of
$X$ with the Fell topology (resp., Vietoris topology). In this
paper, we give the characterization of the space $X$ such that the
hyperspace $(CL(X), \tau_{F})$ (resp., $(CL(X), \tau_{V})$) with a
countable character of closed subsets. We mainly prove that $(CL(X),
\tau_F)$ has a countable character on each closed subset if and only
if $X$ is compact metrizable, and $(CL(X), \tau_F)$ has a countable
character on each compact subset if and only if $X$ is locally
compact and separable metrizable. Moreover, we prove that
$(\mathcal{K}(X), \tau_V)$ have the compact-$G_\delta$ property if
and only if $X$ have the compact-$G_\delta$ property and  every
compact subset of $X$ is metrizable.
\end{abstract}

\maketitle
\section{Introduction}
In this paper, the base space $X$ is always supposed to be $T_1$ and
regular. Let $\mathbb{N}$ and $\omega$ denote the sets of all
positive integers and all non-negative integers, respectively. Given
a topological space $X$, the following collections of subsets of $X$
are the {\it hyperspaces} that we will consider.

\smallskip
$\mathcal{F}_n(X)=\{A\subset X: |A|\leq n, A\neq\emptyset\}$,

\smallskip
$\mathcal{F}(X)=\bigcup\{\mathcal{F}_n(X): n\in \mathbb{N}\}$,

\smallskip
$\mathcal{K}(X)=\{K\subset X: K\ \mbox{is a non-empty, compact
subset in}\ X\}$ and

\smallskip
$CL(X)=\{H\subset X: H\ \mbox{is non-empty, closed in}\ X\}.$

\smallskip
For any open subsets $U_1, \ldots, U_k$ of space $X$, let $$\langle
U_1, ..., U_k\rangle =\{H\in CL(X): H\subset {\bigcup}_{i=1}^k U_i\
\mbox{and}\ H\cap U_j\neq \emptyset, 1\leq j\leq k\}.$$ We endow
$CL(X)$ with {\it Vietoris topology} defined as the topology
generated by $$\{\langle U_1, \ldots, U_k\rangle: U_1, \ldots, U_k\
\mbox{are open subsets of}\ X, k\in \mathbb{N}\}.$$ For any open
subset $U$ of $X$, let $$U^{-}=\{H\in CL(X): H\cap U\neq
\emptyset\}$$ and $$U^{+}=\{H\in CL(X): H\subset U\}.$$ We endow
$CL(X)$ with {\it Fell topology} defined as the topology generated
by the following two families as a subbase:
$$\{U^{-}: U\ \mbox{is any non-empty open in}\ X\}$$ and $$\{(K^{c})^{+}: K\ \mbox{is a compact subset of}\ X, K\neq X\}.$$
We denote the hyperspace with the Vietoris topology and the Fell
topology by $(CL(X), \tau_{V})$ and $(CL(X), \tau_{F})$
respectively. For $T_2$ space, it is well known that $F_1(X)$ is
homeomorphic to $X$ in hyperspaces with the Vietoris topology or the
Fell topology, so we consider all hyperspaces have a closed copy of
$X$.

Let $X$ be a topological space and $A \subseteq X
(\mathcal{A}\subset (CL(X), \tau_V))$ be a subset of $X ((CL(X),
\tau_V))$. Then the \emph{closure} of $A (\mathcal{A})$ in $X
((CL(X), \tau_V))$ is denoted by $\overline{A} (Cl(\mathcal{A}))$,
and $NI(X)$ and $I(X)$ are the sets of all non-isolated points and
isolated points of $X$ respectively. For undefined notations and
terminologies, the reader may refer to \cite{E1989},  \cite{G1984}
and \cite{M1951}.

It is well known that the topics of the hyperspace has been the
focus of much research, see \cite{B1993, HP2002, M1951, NP2015}.
There are many results on the hyperspace $CL(X)$ of a topological
space $X$ equipped with various topologies. In this paper, we endow
$CL(X)$ with the Vietoris topology $\tau_{V}$ and the Fell topology
$\tau_{F}$ respectively. In 1997, Hol\'{a} and Levi in
\cite[Corollary 1.8]{HL1997} gave a characterization of those spaces
$X$ such that $(CL(X), \tau_{V})$ is first countable; in 2003,
Hol\'{a}, Pelant and Zsilinszky in \cite[Theorem 3.1]{HPZ2003}
proved that $(CL(X), \tau_{V})$ is developable iff $(CL(X),
\tau_{V})$ is Moore iff $(CL(X), \tau_{V})$ is metrizable iff
$(CL(X), \tau_{V})$ has a $\sigma$-discrete network iff $X$ is
compact and metrizable. Recently, F. Lin, R. Shen and C. Liu in
\cite{LL2021} considered the following two problems, and gave some
partial answers to Problems~\ref{pr1} and \ref{pr2} respectively.

\begin{problem}\cite[Problem 1.1]{LL2021}\label{pr1}
Let $\mathcal{C}$ be a proper subclass of the class of
first-countable spaces, and let $\mathcal{P}$ be a topological
property. If $(CL(X), \tau_{V})\in\mathcal{C}$, does $X$ have the
property $\mathcal{P}$?
\end{problem}

\begin{problem}\cite[Problem 1.2]{LL2021}\label{pr2}
Let $\mathcal{C}$ be a class of generalized metrizable spaces. If
$(CL(X), \tau_{V})\in\mathcal{C}$, is $X$ compact and metrizable?
\end{problem}

Recall that a space $X$ is a $D_1$-space
\cite{A1966} if every closed subset of $X$ has a countable local
base and $X$ is a $D_0$-space \cite{S1975} if every compact subset
of $X$ has a countable local base. Clearly, each $D_1$-space is a
$D_{0}$-space, each $D_{0}$-space is first-countable, and each Moore
space or space with point-countable base or $\gamma$-space is a
$D_0$-space \cite{G1984}. In \cite{DL1995}, M. Dai and C. Liu
discussed a characterization, some covering properties and the
metrization of $D_{1}$-spaces. In \cite{LL2021},  F. Lin, R. Shen
and C. Liu proved that under (MA+$\neg$CH), $(CL(X), \tau_V)$ is a
$\gamma$-space\footnote{A space $(X, \tau)$ is a {\it
$\gamma$-space} there exists a function $g: \omega\times X\to \tau$
such that (i) $\{g(n, x): n\in \omega\}$ is a base at $x$; (ii) for
each $n\in \omega$ and $x\in X$, there exists $m\in \omega$ such
that $y\in g(m, x)$ implies $g(m, y)\subset g(n, x)$.} if and only
if $(CL(X), \tau_V)$ is a $D_0$-space if and only if $X$ is a
separable metrizable space and $NI(X)$ of $X$ is compact. We listed
some properties of $D_1$-spaces in \cite{DL1995, LL2021} as follows.

\smallskip
{\bf Fact 1:} If $X$ is a $D_1$-space, then $NI(X)$ is countably
compact, see \cite{DL1995};

\smallskip
{\bf Fact 2:} If $X$ is a $D_1$-space, then $X$ is metrizable if and
only if $NI(X)$ is metrizable, see \cite{DL1995}.

\smallskip
{\bf Fact 3:} If $(CL(X), \tau_V)$ is a $D_1$-space, then $X$ is
compact metrizable, see \cite{LL2021}.

It was proved that if $X$ is a Moore space or space with a
point-countable base, so is $(\mathcal{K}(X), \tau_V)$, see
\cite{TM1, TM2} and \cite[Theorem 3.11]{LSL2021} respectively.

Our paper is organized as follows. In Section 2, we mainly discuss
the $D_{1}$-property of the hyperspaces. We first prove that, for a
space $X$, the hyperspace $(CL(X), \tau_{F})$ is a $D_{1}$-space if
and only if $X$ is compact metrizable. Then we prove that, for a
space $X$, $(\mathcal{K}(X), \tau_{V})$ is a $D_{1}$-space if and
only if $(\mathcal{F}_{2}(X), \tau_{V})$ is a $D_{1}$-space if and
only if $X$ is discrete or compact metrizable. In Section 3,  we
mainly discuss the $D_{0}$-property of the hyperspaces. We prove
that, for a space $X$, $(CL(X), \tau_{F})$ is a $D_{0}$-space if and
only if $(\mathcal{F}_{2}(X), \tau_{F})$ is a $D_{0}$-space if and
only if $X$ is locally compact and separable metrizable. We also
prove that, for a space $X$, $(\mathcal{K}(X), \tau_{V})$ is a
$D_{0}$-space if and only if $(\mathcal{F}_{2}(X), \tau_{V})$ is a
$D_{0}$-space if and only if $X$ is a $D_{0}$-space and each compact
subset of $X$ is metrizable. Finally, we discuss the
$G_\delta$-property of hyperspaces and prove that $(\mathcal{K}(X),
\tau_V)$ have the compact-$G_\delta$ property if and only if $X$
have the compact-$G_\delta$ property and  every compact subset of
$X$ is metrizable.

\smallskip
\section{The $D_{1}$-property of the hyperspaces}
In this section, we mainly study the $D_{1}$-property of the
hyperspaces $(CL(X), \tau_{F})$ and $(\mathcal{K}(X), \tau_{F})$ of
those spaces $X$ respectively. First, we give a characterization of
a space $X$ such that $(CL(X), \tau_F)$ is a $D_1$-space. In order
to show that, we need some lemmas.

\begin{lemma}\label{l11111}
\cite[Lemma 2.3.1]{M1951} Let $U_1, ..., U_n$ and $V_1,..., V_m$ be
subsets of a space $X$. Then, in Vietoris topology $(CL(X),
\tau_V)$, we have $\langle U_1, ..., U_n\rangle\subset \langle
V_1,..., V_m\rangle$ if and only if $\bigcup_{j=1}^nU_j\subset
\bigcup_{j=1}^mV_j$ and for every $V_i$ there exists a $U_k$ such
that $U_k\subset V_i$.
\end{lemma}

\begin{lemma}\label{l22222}
\cite[Lemma 2.3.2]{M1951} Let $U_1, ..., U_n$ be subsets of a space
$X$. Then $$Cl(\langle U_1, ..., U_n\rangle)=\langle \overline{U}_1,
..., \overline{U}_n\rangle,$$where $Cl(\langle U_1, ...,
U_n\rangle)$ denotes the closure of the set $\langle U_1, ...,
U_n\rangle$ in $(\mathcal{K}(X), \tau_V)$.
\end{lemma}

We will need the following lemma,  the proof of which we include for
the sake of the completeness.

\begin{lemma}\label{l1}
Let $K$ be a compact subset of a topological space $X$ and $\{U_i:
i\leq k\}$ be an open cover of $K$. Then, for each $i\leq k$, there
exists a compact subset $K_i$ of $X$ such that $K_i\subset U_i$,
$K=\bigcup_{i\leq k}K_i$ and $K_i\neq \emptyset$ whenever $U_i\cap
K\neq \emptyset$.
\end{lemma}

\begin{proof}
For each $x\in K$, pick an open subset $V_x$ such that $x\in
V_x\subset \overline{V}_x\subset U_i$ for some $i\leq k$. Since $K$
is compact and $\{V_x: x\in K\}$ is an open cover of $K$, there is a
finite subcover $\{V_{x_j}: j\leq m\}$ of $K$ such that for each
$i\leq k$ we have $V_{x_j}\subset U_i$ for some $j\leq m$, where
$m\in\mathbb{N}$. For each $i\leq k$, let $$K_i=K\cap
\bigcup\{\overline{V}_{x_j}: \overline{V}_{x_j}\subset U_i\};$$ then
$K_i$ is compact and $K_i\subset U_i$. Clearly, we have
$K=\bigcup\{K_i: i\leq k\}$. Moreover, it is obvious that $K_i\neq
\emptyset$ whenever $U_i\cap K\neq \emptyset$.
\end{proof}

\begin{lemma}\label{l2}
Let $D$ be an infinite countable discrete space. Then
$(\mathcal{F}_2(D), \tau_F)$ is not a $D_1$-space.
\end{lemma}

\begin{proof}
Enumerate $D$ as $\{d_i: i\in \mathbb{N}\}$. In order to prove this
result, we first give the following two Claims.

\smallskip
{\bf Claim 1:} Each $\{d_i\}$ is a non-isolated point in
$(\mathcal{F}_2(D), \tau_F)$.

\smallskip
Suppose not, then there exists $i\in\mathbb{N}$ such that
$\{\{d_i\}\}$ is open in $(\mathcal{F}_2(D), \tau_F)$, hence we can
pick finitely many open subsets $\{U_i: i\leq n\}$ and a compact $K$
of $D$ such that $\bigcap \{U_i^{-}: i\leq n\}\cap (K^c)^{+}\subset
\{\{d_i\}\}$. Clearly, $d_{i}\not\in K$ and $K\cup\{d_{i}\}\neq D$.
Pick any $d\in D\setminus (K\cup\{d_{i}\})$. It easily see that
$\{d, d_i\}\in\bigcap \{U_i^{-}: i\leq n\}\cap (K^c)^{+}$, thus
$\{d, d_i\}\in\{\{d_i\}\}$, that is, $\{d, d_i\}=\{d_i\}$, this is a
contradiction. Hence each $\{d_i\}$ is a non-isolated point in
$(\mathcal{F}_2(D), \tau_F)$.

\smallskip
{\bf Claim 2:}  $\{\{d_i\}: i\in \mathbb{N}\}$ has no cluster point
in $(\mathcal{F}_2(D), \tau_F)$.

\smallskip
Indeed, suppose that $A$ is a cluster point of $\{\{d_i\}: i\in
\mathbb{N}\}$. Pick any $x\in A$; then $\{x\}^-\cap
\mathcal{F}_2(D)$ is a neighborhood of $A$ in $(\mathcal{F}_2(D),
\tau_F)$, but $\{x\}^-$ meets at most one $\{d_i\}$ in
$\mathcal{F}_2(D)$. Thus $A$ is not a cluster point of $\{\{d_i\}:
i\in \mathbb{N}\}$, which is a contradiction.

Now we assume that $(\mathcal{F}_2(D), \tau_F)$ is a $D_1$-space. By
Fact 1, $NI(\mathcal{F}_2(D), \tau_F)$ is countably compact;
however, by Claims 1 and 2, this is impossible since $\{\{d_i\}:
i\in \mathbb{N}\}$ has no cluster point in $(\mathcal{F}_2(D),
\tau_F)$.
\end{proof}

\begin{lemma}\label{l4}
If $(\mathcal{F}_2(X), \tau_F)$ or $(\mathcal{F}_2(X), \tau_V)$ is
perfect\footnote{A space $X$ is called {\it perfect} if every closed
subset of $X$ is a $G_\delta$-set.}, then $X$ has a
$G_\delta$-diagonal\footnote{A space $X$ is said to have a {\it
$G_{\delta}$-diagonal} if, there is a sequence $\{\mathscr{U}_{n}\}$
of open covers of $X$, such that, for each $x\in X$,
$\{x\}=\bigcap_{n\in\mathbb{N}}\mbox{st}(x, \mathscr{U}_{n})$.}.
\end{lemma}

\begin{proof}
We only prove the case for $(\mathcal{F}_2(X), \tau_F)$ since the
case of $(\mathcal{F}_2(X), \tau_V)$ can be shown by a similar way.
Clearly, it suffices to prove that $\Delta=\{(x, x): x\in X\}$ is
the intersection of countably many open subsets of $X^{2}$. Let $f:
X^2 \to (\mathcal{F}_2(X), \tau_F)$ defined by $f((a, b))=\{a, b\}$;
then $f$ is continuous. Indeed, for an open subset $U$ and compact
subset $K$ of $X$, we have
$$f^{-1}(U^-\cap \mathcal{F}_2(X))=(U\times X)\cup (X\times U)$$ and
$$f^{-1}((K^c)^+\cap \mathcal{F}_2(X))=(X\setminus K)\times
(X\setminus K),$$ which are open in $X^2$. Hence $f$ is continuous.
Since $(\mathcal{F}_2(X), \tau_F)$ is perfect and
$\mathcal{F}_1(X)=\{\{x\}: x\in X\}$ is a closed subset of $(F_2(X),
\tau_F)$, it follows that $\mathcal{F}_1(X)=\bigcap_{n\in
\mathbb{N}} \mathcal{G}_n$, where $\mathcal{G}_n$ is open in
$(\mathcal{F}_2(X), \tau_F)$ for each $n\in\mathbb{N}$. Clearly,
$\Delta=f^{-1}(\mathcal{F}_1(X))=\bigcap_{n\in \mathbb{N}}
f^{-1}(\mathcal{G}_n)$. Hence $X$ has a $G_\delta$-diagonal.
\end{proof}

Now we can prove one of main theorems in this section.

\begin{theorem}\label{t4}
The following statements are equivalent for a space $X$.
\begin{enumerate}
 \item $(CL(X), \tau_F)$ is a $D_1$-space;

\smallskip
 \item $(\mathcal{K}(X), \tau_F)$ is a $D_1$-space;

\smallskip
 \item $(\mathcal{F}_n(X), \tau_F)$ is a $D_1$-space for some $n\geq 2$;

\smallskip
 \item $X$ is compact metrizable.
\end{enumerate}
\end{theorem}

\begin{proof}
Note that a $D_1$-space is a property closed under taking closed
subspaces, (1) $\Rightarrow$ (3) and (2) $\Rightarrow$ (3) are
trivial. Note that if $X$ is compact, then $(\mathcal{K}(X),
\tau_F)=(CL(X), \tau_F)\cong (CL(X), \tau_V)=(\mathcal{K}(X),
\tau_V)$, the implications (4) $\Rightarrow$ (1) and (4)
$\Rightarrow$ (2) follow from \cite{M1951}. We only need to prove
(3) $\Rightarrow$ (4).

It suffices to consider the case when $n=2$. Assume that
$(\mathcal{F}_2(X), \tau_F)$ is a $D_1$-space. Then $X$ is a
$D_1$-space. By Fact 1 and Lemma ~\ref{l4}, $NI(X)$ is a countably
compact subspace with a $G_{\delta}$-diagonal. Then $NI(X)$ is
metrizable \cite[Theorem 2.14]{G1984}, hence it is compact. Now we prove that $X$
is compact. Indeed, let $\mathcal{U}$ be an any open cover of $X$;
then we can find a finite subfamily $\mathcal{U}'\subset
\mathcal{U}$ such that $NI(X)\subset \bigcup\mathcal{U}'$. Let
$D=X\setminus \bigcup\mathcal{U}'$; then $D\subset I(X)$. We claim
that $|D|<\omega$. Otherwise, without loss of generality, $D$ is an
infinite countable closed discrete subset of $X$, then it follows
from Lemma ~\ref{l2} that $(\mathcal{F}_2(D), \tau_F)$ is not a
$D_1$-space, this is a contradiction. Therefore, $X$ is covered by
finitely many elements of $\mathcal{U}$, thus $X$ is compact.

By Lemma~\ref{l4}, $X$ has a $G_\delta$-diagonal, then $X$ is
metrizable by \cite[Theorem 2.13]{G1984}.
\end{proof}

\begin{remark}
However, if $X$ is a discrete space, then it is easily verified that
$(\mathcal{K}(X), \tau_V)$ is a discrete space, hence
$(\mathcal{F}_n(X), \tau_V)$ is discrete for any $n\in\mathbb{N}$,
thus all are $D_1$-spaces. Therefore, Theorem~\ref{t4} does not hold
for the case of the hyperspace with Vietoris topology. However, we
have the following theorem.
\end{remark}

In \cite{LL2021}, we have proved that for a space $X$ if $(CL(X),
\tau_V)$ is a $D_1$-space then $X$ is compact and metrizable.
Therefore, it is natural to characterize $X$ such that the subspace
$(\mathcal{K}(X), \tau_V)$ of $(CL(X), \tau_V)$ is a $D_1$-space,
see the following theorem.

\begin{theorem}\label{t6}
The following statements are equivalent for a space $X$.
\begin{enumerate}
\item $(\mathcal{K}(X), \tau_V)$ is a $D_1$-space;

\smallskip
\item $(\mathcal{F}_n(X), \tau_V)$ is a $D_1$-space for some $n\geq 2$;

\smallskip
\item $X$ is discrete or compact metrizable.
\end{enumerate}
\end{theorem}

\begin{proof}
The implication of (1) $\Rightarrow$ (2) is trivial since a closed
subspace of a $D_1$-space is a $D_1$-space. It suffices to prove
that (3) $\Rightarrow$ (1) and (2) $\Rightarrow$ (3).

(3) $\Rightarrow$ (1). If $X$ is discrete, then it is easy to see
that $(\mathcal{K}(X), \tau_V)$ is discrete; if $X$ is compact
metrizable, then $(\mathcal{K}(X), \tau_V)$ is compact metrizable by
\cite{M1951}. Hence $(\mathcal{K}(X), \tau_V)$ is a $D_1$-space.

(2) $\Rightarrow$ (3). We only consider the case for $n=2$. Assume
$(\mathcal{F}_2(X), \tau_V)$ is a $D_1$-space. Then $X$ is a
$D_1$-space. By fact 1, $NI(X)$ is countably compact, then it
follows from Lemma~\ref{l4} that $X$ is metrizable. Suppose $X$ is
neither discrete nor compact, there exist a closed infinite
countable discrete subset $\{d_i: i\in \mathbb{N}\}\subset I(X)$ and
a non-trivial sequence $\{x_n\in \mathbb{N}\}$ of $X$ converging to
$x\in NI(X)$. Without loss of generality, we may assume that $\{d_i:
i\in \mathbb{N}\}\cap\{x_n: n\in \mathbb{N}\}=\emptyset$, $d_n\neq
d_{m}$ and $x_n\neq x_{m}$ for any $n\neq m$. Clearly, for each
$i\in \mathbb{N}$, the set $\{d_i, x\}\in NI(\mathcal{F}_2(X),
\tau_V)$ since $\{d_i, x_n\}\to \{d_i, x\}$ in $(\mathcal{F}_2(X),
\tau_V)$ as $n\to \infty$. In order to obtain a contradiction, we
prove the following Claim 3. Indeed, since $(\mathcal{F}_2(X),
\tau_V)$ is a $D_1$-space, it follows that $NI(\mathcal{F}_2(X),
\tau_V)$ is countably compact; however, the set $\{\{d_i, x\}: i\in
\mathbb{N}\}$ is discrete in $NI(\mathcal{F}_2(X), \tau_V)$.

\smallskip
{\bf Claim 3} The set $\{\{d_i, x\}: i\in \mathbb{N}\}$ is discrete
in $NI(\mathcal{F}_2(X), \tau_V)$.

\smallskip
Take any $K\in NI(\mathcal{F}_2(X), \tau_V)$. If $|K|=2$, we write
$K=\{a_1, a_2\}$, and let $V_j$ be a neighborhood of $a_j$ for each
$j\leq 2$ such that $|(V_1\cup V_{2})\cap \{d_i: i\in
\mathbb{N}\}|\leq 1$ and $V_1\cap V_2=\emptyset$ (this is possible
since $K\in NI(\mathcal{F}_2(X), \tau_V)$ which implies that at
least one of the points $\{a_{1}, a_{2}\}$ is not a non-isolated
point in $X$). Then it is easily verified that $|\langle V_1,
V_2\rangle\cap \{\{d_i, x\}: i\in \mathbb{N}\}|\leq 1$. If $|K|=1$,
then let $K=\{a\}$ and let $V$ be an open neighborhood of $a$ in $X$
with $|V\cap \{d_i: i\in \mathbb{N}\}|\leq 1$. Then $\langle
V\rangle$ is a neighborhood of $K$ and $|\langle V\rangle \cap
\{\{d_i, x\}: i\in \mathbb{N}\}|\leq 1$. Therefore, $\{\{d_i, x\}:
i\in \mathbb{N}\}$ is discrete in $NI(\mathcal{F}_2(X), \tau_V)$.
\end{proof}

However, for the cases of $(\mathcal{F}(X), \tau_V)$ and
$(\mathcal{F}(X), \tau_F)$, the situations are different, see the
following two theorems.

\begin{theorem}\label{t9}
Let $X$ be a space. Then $(\mathcal{F}(X), \tau_V)$ is a $D_1$-space
if and only if $X$ is discrete.
\end{theorem}

\begin{proof}
If $X$ is discrete, it is easy to see that $(\mathcal{F}(X),
\tau_V)$ is discrete, hence it is a $D_1$-space.

If $(\mathcal{F}(X), \tau_V)$ is a $D_1$-space, then
$(\mathcal{F}_2(X), \tau_V)$ is a $D_1$-space, by Theorem ~\ref{t6},
$X$ is discrete or compact metrizable. If $X$ is compact metrizable
and non-discrete, then it contains a convergent sequence
$S=\{x\}\cup\{x_n: n\in \mathbb{N}\}$ with $x_n\to x$ as
$n\rightarrow\infty$, where $x$ is a non-isolated point in $X$. Let
$\mathcal{A}=\{\{x\}\cup\{x_i: i\leq n\}: n\in \mathbb{N}\}$. Then
it is obvious that $\mathcal{A}\subset NI(\mathcal{F}(S))$. We claim
that $\mathcal{A}$ is discrete in $\mathcal{F}(S)$.

Suppose not, $\mathcal{A}$ has a cluster point $B$ in
$\mathcal{F}(S)$. Let $B=\{y_1, ..., y_k\}$. If $x\notin B$, then
$\langle \{y_1\}, ..., \{y_k\}\rangle$ is an open neighborhood of
$B$, $|\langle \{y_1\}, ..., \{y_k\}\rangle\cap
\mathcal{A}|=\emptyset$, this is a contradiction. If $x\in B$, say
$y_1=x$, pick an open neighborhood $V=\{x\}\cup\{x_{n}: n>k+1\}$ of
$y_1$ in $S$. Clearly, $|S\setminus V|>k+1$. Then $\langle V,
\{y_2\}, ..., \{y_k\}\rangle$ is an open neighborhood of $B$. Since
$\langle V, \{y_2\}, ..., \{y_k\}\rangle\cap\{\{x\}\cup\{x_i: i\leq
n\}: n>k+1\}=\emptyset$. Indeed, if $\{x, x_1, ..., x_m\}\in \langle
V, \{y_2\}, ..., \{y_k\}\rangle\cap\{\{x\}\cup\{x_i: i\leq n\}:
n>k+1\}$ for some $m>k+1$, $\{x, x_1, ..., x_m\}\subset V\cap \{y_2,
... y_k\}$, it implies that $\{x_1, ..., x_{k+1}\}\subset \{y_2,
..., y_k\}$, which is impossible. It follows that $|\langle V,
\{y_2\}, ..., \{y_k\}\rangle\cap \mathcal{A}|<\omega$, which is a
contradiction.

Therefore, $\mathcal{A}$ is discrete in $\mathcal{F}(S)$. From fact
1, $\mathcal{F}(S)$ is not a $D_1$-space, which implies that
$(\mathcal{F}(X), \tau_V)$ is not a $D_1$-space, this is a
contradiction. Hence $X$ is not compact metrizable. Thus $X$ is
discrete.
\end{proof}

\begin{theorem}
Let $X$ be a space. Then $(\mathcal{F}(X), \tau_F)$ is a $D_1$-space
if and only if $X$ is finite.
\end{theorem}

\begin{proof}
If $X$ is finite, then it is easy to see that $(\mathcal{F}(X),
\tau_F)$ is a $D_1$-space.

If $(\mathcal{F}(X), \tau_F)$ is a $D_1$-space, then
$(\mathcal{F}_2(X), \tau_F)$ is a $D_1$-space, hence $X$ is compact
metrizable by Theorem ~\ref{t4}. From \cite[Exercise 5.1, problem
3]{B1993}, it follows that $$(CL(X), \tau_V)=(CL(X), \tau_F),$$ then
$(\mathcal{F}(X), \tau_V)=(\mathcal{F}(X), \tau_F)$ is a
$D_1$-space, which shows that $X$ is discrete by Theorem ~\ref{t9},
hence $X$ is finite since a discrete compact space is finite.
\end{proof}

As some applications of above results, we have the following remark.

\begin{remark}
Let $X$ be an arbitrary infinite compact metrizable space. Then
$(CL(X), \tau_V)$ is a compact metrizable space, thus $(CL(X),
\tau_V)$ is a $D_{1}$-space and $(\mathcal{F}(X), \tau_V)$ is
metrizable; however, $(\mathcal{F}(X), \tau_V)$ is not a
$D_{1}$-space by Theorem~\ref{t9}. Therefore, there exists a closed
subset $\mathbf{F}$ of $(\mathcal{F}(X), \tau_V)$ such that
$\mathbf{F}$ has no countable character at $(\mathcal{F}(X),
\tau_V)$. Moreover, $(\mathcal{K}(X), \tau_V)$ is a $D_1$-space;
however, $(\mathcal{F}(X), \tau_F)$ is not a $D_1$-space.
\end{remark}

Finally we consider the space $X$ such that the subspace
$(\mathcal{K}(X)\setminus \mathcal{F}(X), \tau_V)$ is a
$D_{1}$-space. First, we need two lemmas.

A subset $P$ of $X$ is called a \emph{sequential neighborhood} of $x
\in X$, if each sequence converging to $x$ is eventually\footnote{A
sequence $\{x\}\cup\{x_i: i\in \mathbb{N}\}$ with $x_i\to x$ is
called {\it eventually} in some subset $P$ if there exists
$k\in\mathbb{N}$ such that $\{x\}\cup\{x_i: i\geq k\}\subset P$.} in
$P$. A subset $U$ of $X$ is called \emph{sequentially open} if $U$
is a sequential neighborhood of each of its points. A subset $F$ of
$X$ is called \emph{sequentially closed} if $X\setminus F$ is
sequentially open. The space $X$ is called a \emph{sequential space}
if each sequentially open subset of $X$ is open.

\begin{lemma}\label{l111}
Let $X$ be a sequential space. Then $(\mathcal{F}(X), \tau_V)$ is
open in $(\mathcal{K}(X), \tau_V)$ if and only if $X$ is discrete.
\end{lemma}

\begin{proof}
It suffices to prove the necessity. Assume that $(\mathcal{F}(X),
\tau_V)$ is open in $(\mathcal{K}(X), \tau_V)$, and assume that $X$
is not discrete. Then there exists a non-trivial sequence
$\{x_{n}\}$ converging to $x$ as $n\rightarrow\infty$. Since
$\{x\}\in\mathcal{F}(X)$, there exists an open neighborhood $U$ of
$x$ such that $\langle U\rangle\subset \mathcal{F}(X)$. However,
since $x\in U$ and $\{x_{n}\}$ converging to $x$ as
$n\rightarrow\infty$, we can find $m\in\mathbb{N}$ such that
$\{x\}\cup\{x_{n}: n\geq m\}\subset U$, then $\{x\}\cup\{x_{n}:
n\geq m\}\in\langle U\rangle\subset\mathcal{F}(X)$, which is a
contradiction since $\{x\}\cup\{x_{n}: n\geq m\}\not\in
\mathcal{F}(X)$.
\end{proof}

\begin{lemma}\label{l112}
Let $X$ be a space. Then any point of $(\mathcal{K}(X)\setminus
\mathcal{F}(X), \tau_V)$ is not isolated.
\end{lemma}

\begin{proof}
Take any $K\in \mathcal{K}(X)\setminus \mathcal{F}(X)$. Then $K$ is
an infinite compact subset of $X$. Suppose that $K$ is an isolated
point in $(\mathcal{K}(X)\setminus \mathcal{F}(X), \tau_V)$. Then
there exists a basic neighborhood $\langle U_{1}, \ldots,
U_{k}\rangle$ of $K$ in $(\mathcal{K}(X), \tau_V)$ such that
$\langle U_{1}, \ldots, U_{k}\rangle\cap (\mathcal{K}(X)\setminus
\mathcal{F}(X))=\{K\}$, where the family $\{U_{i}: i\leq k\}$ is a
disjoint collection of open subsets of $X$. Clearly,
$K=\bigcup_{i=1}^{k}U_{i}$; otherwise, we can pick any point $a\in
\bigcup_{i=1}^{k}U_{i}\setminus K$, then $K\cup\{a\}\in\langle
U_{1}, \ldots, U_{k}\rangle\cap (\mathcal{K}(X)\setminus
\mathcal{F}(X))$, which is a contradiction. Since $K$ is infinite,
there exists $j\leq k$ such that $|K\cap U_{j}|\geq\omega$. Since
$X$ is Hausdorff, there exist a point $b\in K\cap U_{j}$ and an open
neighborhood $V\subset U_{j}$ of $b$ such that $|(K\setminus V)\cap
U_{j}|\geq\omega$, hence $K\setminus V\in\mathcal{K}(X)\setminus
\mathcal{F}(X)$, which is a contradiction. Therefore, $K$ is an
isolated point in $(\mathcal{K}(X)\setminus \mathcal{F}(X),
\tau_V)$.
\end{proof}

Since Fell topology is coarser than the Vietoris topology, it
follows from Lemmas~\ref{l111} and ~\ref{l112} that we have the
following two corollaries.

\begin{corollary}\label{c111}
Let $X$ be a sequential space. Then $(\mathcal{F}(X), \tau_F)$ is
open in $(\mathcal{K}(X), \tau_F)$ if and only if $X$ is discrete.
\end{corollary}

\begin{corollary}\label{c112}
Let $X$ be a space. Then any point of $(\mathcal{K}(X)\setminus
\mathcal{F}(X), \tau_F)$ is not isolated.
\end{corollary}

Now we can prove the following theorem.

\begin{theorem}\label{tt40}
Let $X$ be a sequential space with a
$G_{\delta}^{\ast}$-diagonal\footnote{A space $X$ is said to have a
{\it $G_{\delta}^{\ast}$-diagonal} if, there is a sequence
$\{\mathscr{U}_{n}\}$ of open covers of $X$, such that, for each
$x\in X$, $\{x\}=\bigcap_{n\in\mathbb{N}}\overline{\mbox{st}(x,
\mathscr{U}_{n})}$.}. Then $(\mathcal{K}(X)\setminus \mathcal{F}(X),
\tau_V)$ is a $D_{1}$-space if and only if $X$ is discrete.
\end{theorem}

\begin{proof}
It suffices to prove the necessity. Suppose that
$(\mathcal{K}(X)\setminus \mathcal{F}(X), \tau_V)$ is a
$D_{1}$-space, and suppose that $X$ is non-discrete. Then
$\mathcal{K}(X)\setminus \mathcal{F}(X)$ is countably compact by
Fact 1 and Lemma~\ref{l112}. Since $X$ has a
$G_{\delta}^{\ast}$-diagonal, it follows that $(\mathcal{K}(X),
\tau_V)$ has a $G_{\delta}$-diagonal by \cite[Theorem 1]{TM1}, thus
$\mathcal{K}(X)\setminus \mathcal{F}(X)$ has a $G_{\delta}$-diagonal
too. Therefore, $\mathcal{K}(X)\setminus \mathcal{F}(X)$ is compact
by \cite[Theorem 2.14]{G1984}, then $\mathcal{K}(X)\setminus \mathcal{F}(X)$ is
closed in $(\mathcal{K}(X), \tau_V)$, which shows that
$\mathcal{F}(X)$ is open in $(\mathcal{K}(X), \tau_V)$. However, by
Lemma~\ref{l111}, $\mathcal{F}(X)$ is not open in $(\mathcal{K}(X),
\tau_V)$, which is a contradiction. Therefore, $X$ is discrete.
\end{proof}

However, it is still unknown for us if Theorem~\ref{tt40} holds for
the case of the Fell topology. Therefore, we have the following two
questions.

\begin{question}
Let $X$ be a sequential space with a $G_{\delta}^{\ast}$-diagonal.
If $(\mathcal{K}(X)\setminus \mathcal{F}(X), \tau_F)$ is a
$D_{1}$-space, is $X$ discrete?
\end{question}

\begin{question}
Let $X$ be a space. Give a characterization of $X$ such that
$(\mathcal{K}(X), \tau_F)$ or $(CL(X), \tau_F)$ has a
$G_{\delta}$-diagonal.
\end{question}

The following theorem gives a characterization of a space $X$ such
that $(CL_{\emptyset}(X), \tau_F)$ is a $D_1$-space, where
$CL_{\emptyset}(X)=CL(X)\cup\{\emptyset\}$.

\begin{theorem}\label{t13}
The following statements are equivalent for space $X$.
\begin{enumerate}
\item $(CL_{\emptyset}(X), \tau_F)$ is a $D_1$-space;

\smallskip
\item $(\mathcal{F}_n(X)\cup\{\emptyset\}, \tau_F)$ is a $D_1$-space for some $n\geq 2$;

\smallskip
\item $X$ is locally compact, separable and metrizable.
\end{enumerate}
\end{theorem}

\begin{proof}
By \cite[Theorem 5.1.5]{B1993}, we have (3) $\Rightarrow$ (1). The
implication of (1) $\Rightarrow$ (2) is obvious. It suffices to
prove that (2) $\Rightarrow$ (3). Assume
$(\mathcal{F}_n(X)\cup\{\emptyset\}, \tau_F)$ is a $D_1$-space for
some $n\geq 2$, then $(\mathcal{F}_2(X)\cup\{\emptyset\}, \tau_F)$
is a $D_0$-space, hence $(\mathcal{F}_2(X), \tau_F)$ is a
$D_0$-space, by Theorem ~\ref{t10}, $X$ is locally compact,
separable and metrizable.
\end{proof}

\smallskip
\section{The $D_{0}$-property of hyperspaces}
In this section, we mainly discuss the $D_{0}$-property of the
hyperspaces. We first give a characterization of a space $X$ such
that $(CL(X), \tau_F)$ is a $D_0$-space.

\begin{theorem}\label{t10}
The following statements are equivalent for space $X$.

\begin{enumerate}
\item $(CL(X), \tau_F)$ is metrizable;

\smallskip
\item $(CL(X), \tau_F)$ is a $\gamma$-space;

\smallskip
\item $(CL(X), \tau_F)$ is a $D_0$-space;

\smallskip
\item $(\mathcal{K}(X), \tau_F)$ is a $D_0$-space;

\smallskip
\item $(\mathcal{F}(X), \tau_F)$ is a $D_0$-space;

\smallskip
\item $(\mathcal{F}_2(X), \tau_F)$ is a $D_0$-space;

\smallskip
\item $X$ is locally compact, separable and metrizable.
\end{enumerate}
\end{theorem}

\begin{proof}
The implications (1) $\Rightarrow$ (2) $\Rightarrow$ (3)
$\Rightarrow$ (4) $\Rightarrow$ (5) $\Rightarrow$ (6) are trivial.
(7) $\Leftrightarrow$ (1) by \cite[Theorem 5.1.5]{B1993}. We only
prove (6) $\Rightarrow$ (7).

Assume that $(\mathcal{F}_2(X), \tau_F)$ is a $D_0$-space, then $X$
is a $D_0$-space. Fix any $x\in X$; then let $\{V_n^-\cap (K_{n,
x}^c)^+\cap \mathcal{F}_2(X): n\in \mathbb{N}\}$ be a countable
local base at $\{x\}$ in $(\mathcal{F}_2(X), \tau_F)$, where each
$V_{n}$, $K_{n, x}$ is open,  compact in $X$ respectively. For any
compact $K\in \mathcal{K}(X\setminus \{x\})$, $(K^c)^+$ is an open
neighborhood of $\{x\}$, there exists $n\in\mathbb{N}$ such that
$V_n^-\cap (K_{n, x}^c)^+\cap \mathcal{F}_2(X)\subset (K^c)^+\cap
\mathcal{F}_2(X)$, then $K\subset K_{n, x}$. Therefore, $\{K_{n, x}:
n\in \mathbb{N}\}$ is cofinal in $\mathcal{K}(X\setminus \{x\})$.

\smallskip
{\bf Claim 4} $X$ is locally compact.

Fix $x\in X$, pick any $y\neq x$ and let $\{U_n: n\in \mathbb{N}\}$
be a decreasingly local base at $x$ in $X$ with $y\notin U_n$ for
each $n\in\mathbb{N}$. We claim that there exist $n\in\mathbb{N}$
and $k\in\mathbb{N}$ such that $U_k\subset K_{n, y}$; otherwise, for
any $n\in\mathbb{N}$, pick $x_n\in U_n\setminus K_{n, y}$; then
$x_n\to x$ as $n\rightarrow\infty$ in $X$ and $S=\{x\}\cup\{x_n:
n\in \mathbb{N}\}\in \mathcal{K}(X\setminus \{y\})$, then $S\subset
K_{m, y}$ for some $m\in\mathbb{N}$ since $\{K_{n, y}: n\in
\mathbb{N}\}$ is cofinal in $\mathcal{K}(X\setminus \{y\})$. This is
a contradiction. Hence $X$ is locally compact.

Fix any compact subset $K$ of $X$. Since $(\mathcal{F}_2(K),
\tau_F)$ is a compact $D_0$-space,  then $(\mathcal{F}_2(K),
\tau_F)$ is perfect. By Lemma ~\ref{l4} and \cite[Theorem 2.13]{G1984}, $K$ is
metrizable. Hence $X$ locally metrizable by Claim 4. By the above
proof, $X$ is also a $\sigma$-compact space, thus it is separable
and locally compact metrizable.
\end{proof}

\begin{remark}
By Theorem~\ref{t10}, $(CL(\mathbb{R}), \tau_F)$ is a $D_{0}$-space
(indeed, it is metrizable, thus perfect normal), where $\mathbb{R}$
is the real number with the usual topology; however,
$(CL(\mathbb{R}), \tau_F)$ is not a $D_{1}$-space by
Theorem~\ref{t4}.
\end{remark}

A space $(X, \tau)$ is a {\it $\gamma$-space} if there exists a
function $g: \omega\times X \to \tau$ such that (i) $\{g(n, x): n\in
\omega\}$ is a base at $x$; (ii) for each $n\in \omega$ and $x\in
X$, there exists $m\in \omega$ such that $y\in g(m, x)$ implies
$g(m, y)\subset g(n, x)$. By \cite[Theorem 10.6(iii)]{G1984}, each
$\gamma$-space is a $D_0$-space.

In \cite{LL2021},  F. Lin, R. Shen and C. Liu proved that under
(MA+$\neg$CH), $(CL(X), \tau_V)$ is a $\gamma$-space if and only if
$(CL(X), \tau_V)$ is a $D_0$-space if and only if $X$ is a separable
metrizable space and $NI(X)$ is compact. Therefore, we have the
following question.

\begin{question}
If $(CL(X), \tau_V)$ is a $D_0$-space, is $(CL(X), \tau_V)$ a
$\gamma$-space?
\end{question}

For any non-separable and metrizable space $X$, it follows from
\cite[Theorem 5.17]{LL2021} and the following Theorem~\ref{tt44}
that $(CL(X), \tau_V)$ is not a $\gamma$-space and $(\mathcal{K}(X),
\tau_V)$ is a $D_0$-space.

\begin{lemma}\label{l3}\cite{M1973}
The following are equivalent for a space $X$.
\begin{enumerate}
\item $X$ is an open compact-covering\footnote{A continuous map $f: X\rightarrow Y$ is called {\it compact-covering} if each compact subset of $Y$ is the image of some compact subset of $X$.} image of a metric space;

\smallskip
\item Each compact subset of $X$ is metrizable and has a countable
neighborhood base in $X$;

\smallskip
\item Every compact subset of $X$ has a countable outer base\footnote{A collection $\mathcal{B}$ of open subsets of a space $X$ is called
an {\it outer base} \cite{M1973} of a subset $A$ in $X$, if there
exists $B_{x}\in \mathcal{B}$ such that $x\in B_{x}\subset U$ for
every $x\in A$ and an open neighborhood $U$ of $x$ in $X$.} in $X$.
\end{enumerate}
\end{lemma}

\begin{theorem}\label{tt44}
The following statements are equivalent for a space $X$.
\begin{enumerate}
\item $(\mathcal{K}(X), \tau_V)$ is an open compact-covering image of a metric space.

\smallskip
 \item $(\mathcal{K}(X), \tau_V)$ is a $D_0$-space;

\smallskip
\item $(\mathcal{F}(X), \tau_V)$ is a $D_0$-space;

\smallskip
\item $(\mathcal{F}_2(X), \tau_V)$ is a $D_0$-space;

\smallskip
\item $X$ is a $D_0$-space and every compact subset of $X$ is metrizable.
\end{enumerate}
\end{theorem}

\begin{proof}
By Lemma ~\ref{l3}, we have (1) $\Rightarrow$ (2). The implications
(2) $\Rightarrow$ (3) $\Rightarrow$ (4) are trivial. we only need to
prove (4) $\Rightarrow$ (5), (5) $\Rightarrow$ (2) and (2)
$\Rightarrow$ (1).

(4) $\Rightarrow$ (5). Assume that $(\mathcal{F}_{2}(X), \tau_V)$ is
a $D_0$-space, then $X$ is a $D_0$-space. For any compact subset $H$
of $X$, $(\mathcal{F}_2(H), \tau_V)$ is a compact, $D_0$-subspace of
$(\mathcal{F}_{2}(X), \tau_V)$ \cite{M1951}, then it is perfect. By
Lemma ~\ref{l4}, $H$ has a $G_\delta$-diagonal, then $H$ is
metrizable by \cite[Theorem 2.14]{G1984}.

\smallskip
(5) $\Rightarrow$ (2). Fix any compact subset $\mathbf{K}$ of
$(\mathcal{K}(X), \tau_V)$. We prove that $\mathbf{K}$ has a
countable base in $(\mathcal{K}(X), \tau_V)$.

Indeed, let $K=\bigcup \mathbf{K}$. Then $K$ is compact in $X$ by
\cite[Theorem 2.5.2]{M1951}, hence $K$ has a countable base and is
metrizable. By Lemma ~\ref{l3}, $K$ has a countable outer base
$\mathcal{B}$. Let $\Gamma=\{\langle \mathcal{B}'\rangle:
\mathcal{B}'\in \mathcal{B}^{<\omega}\}$, $\Delta=\{\bigcup
\mathcal{C}: \mathcal{C}\in \Gamma^{<\omega}\}$; then $|\Delta|\leq
\omega$. Now it suffices to prove that $\Delta$ is a countable open
base in $(\mathcal{K}(X), \tau_V)$.

Take any open neighborhood $\widehat{U}$ of $\mathbf{K}$ in
$(\mathcal{K}(X), \tau_V)$; then there exist finitely many basic
open neighborhood $\{\widehat{U}_i: i\leq l\}$ such that
$\mathbf{K}\subset \bigcup\{\widehat{U}_i: i\leq l\}\subset
\widehat{U}$, where we may assume that $\{\widehat{U}_i: i\leq l\}$
is a minimal cover of $\mathbf{K}$. By Lemma ~\ref{l1}, there exist
finitely many compact subsets $\{\mathbf{K}_i: i\leq l\}$ of
$(\mathcal{K}(X), \tau_V)$ such that $\mathbf{K}=\bigcup_{i\leq
l}\mathbf{K}_i$ and $\mathbf{K}_i\subset \widehat{U}_i$ for each
$i\leq l$. We will prove that each $\mathbf{K}_{i}$ is contained in
an element of $\Delta$ itself contained in $\widehat{U}$, which
shows that $\mathbf{K}$ is contained in an element of $\Delta$ that
is contained in $\widehat{U}$. Therefore, without loss of
generality, we may assume $\widehat{U}=\langle U_1, ...,
U_k\rangle$, where each $U_i$ $(i\leq k)$ is an open subset in $X$.

Since $(\mathcal{K}(X), \tau_V)$ is regular, there exist finitely
many open subsets $\{\widehat{V}(i): i\leq n\}$ of $(\mathcal{K}(X),
\tau_V)$ such that $$\mathbf{K}\subset \bigcup_{i\leq
n}\widehat{V}(i)\subset \bigcup_{i\leq n}Cl(\widehat{V}(i))\subset
\widehat{U},$$ where $\widehat{V}(i)=\langle V_1(i), ...,
V_{r_i}(i)\rangle$, each $r_{i}\in\mathbb{N}$ and each $V_j(i)$ is
open in $X$.

Fix any $i\leq n$. Then $Cl(\widehat{V}(i))\subset \widehat{U}$. For
each $t\leq k$, it follows from Lemmas~\ref{l11111} and~\ref{l22222}
that $U_t\supset \overline{V_{i(t)}(i)}$ for some $i(t)\leq r_i$.
Let $A=\{i(t): t\leq k\}$ and $B=\{1, ..., r_i\}\setminus A$. Put
$$K(i)=\bigcup\{K: K\in \mathbf{K}\cap Cl(\widehat{V}(i))\}.$$ Then
$K(i)$ is compact in $X$ by \cite[Theorem 2.5.2]{M1951}.

For any $K'\in \{K: K\in \mathbf{K}\cap Cl(\widehat{V}(i))\}$, we have $K'\cap
\overline{V_j(i)}\neq \emptyset$ for $j\leq r_i$ and $K'\subset
\cup\{\overline{V_j(i)}: j\leq r_i\}$. It is obvious that
$K(i)\subset \cup\{\overline{V_j(i)}: j\leq r_i\}$.  Note that
$K'\subset K(i)$, then $K(i)\cap \overline{V_j(i)}\neq \emptyset$
for each $j$. Then, by Lemma~\ref{l22222}, we have
$$K(i)\in Cl(\widehat{V}(i)=\langle \overline{V_1(i)}, ...,
\overline{V_{r_i}(i)}\rangle.$$

For each $j\in A$,
$$K(i)\cap \overline{V_j(i)}\subset \bigcap\{U_t: t\leq k,
i(t)=j\}=W_j.$$ For $x\in K(i)\cap \overline{V_j(i)}$, pick $B_x\in
\mathcal{B}$ such that $x\in B_x\subset W_j$. Then there is a finite
subfamily $\mathcal{B}(j)$ of $\{B_x: x\in K(i)\cap
\overline{V_j(i)}\}$ such that $$K(i)\cap \overline{V_j(i)}\subset
\bigcup\mathcal{B}(j)\subset W_j.$$ For any $j\in B$, choose a
finite subfamily $\mathcal{B}(j)\subset \mathcal{B}$ such that
$$K(i)\cap \overline{V_j(i)}\subset \bigcup\mathcal{B}(j)\subset
\bigcup_{t\leq k}U_t.$$ Now put $$\Re(i)=\{\bigcup_{j\leq
r_i}\mathcal{B}'(j): \emptyset\neq \mathcal{B}'(j)\subset \mathcal{B}(j), j\leq r_i\}.$$
Clearly, $|\Re(i)|<\omega$. Then $$\mathbf{K}\cap
Cl(\widehat{V}(i))\subset \bigcup\{\langle \mathcal{W}\rangle:
\mathcal{W}\in \Re(i)\}\subset \widehat{U},$$ where $\langle
\mathcal{W}\rangle=\langle B_1, ..., B_m\rangle$ for
$\mathcal{W}=\{B_1, ..., B_m\}$. Indeed, for any $K\in
\mathbf{K}\cap Cl(\widehat{V}(i))$, then $K\subset K(i)$. For each
$j\leq r_{i}$, pick $\mathcal{C}(j)\subset \mathcal{B}(j)$ such that
$K\cap \overline{V_j(i)}\cap C\neq \emptyset$ for each $C\in
\mathcal{C}(j)$ and $K\cap \overline{V_j(i)}\subset
\bigcup\mathcal{C}(j)$. Enumerate $\bigcup_{j\leq
r_i}\mathcal{C}(j)$ as $\{C_1, ..., C_p\}$; then
$\mathcal{W}_{i}=\{C_1, ..., C_p\}\in \Re(i)$. Then $K\cap C_{l}\neq
\emptyset$ for each $l\leq p$, $K\subset \bigcup_{i\leq p}C_i$.
Therefore, $K\in \langle \mathcal{W}_{i}\rangle\subset \widehat{U}$.

Hence $$\mathbf{K}=\bigcup_{i\leq n}(\mathbf{K}\cap
Cl{\widehat{V}(i)})\subset \bigcup\{\langle \mathcal{W}_{i}\rangle:
\mathcal{W}_{i}\in \Re(i), i\leq n\}\subset \widehat{U}.$$ Note that
$\Re(i)\subset \mathcal{B}$ for $i\leq n$ and $|\Delta|\leq \omega$.
Hence $\mathbf{K}$ has a countable base in $(\mathcal{K}(X),
\tau_V)$.

(2) $\Rightarrow$ (1). By Lemma ~\ref{l3}, we only prove that each
compact subset $\mathbf{H}$ of $(\mathcal{K}(X), \tau_V)$ is
metrizable. Let $H=\bigcup \mathbf{H}$. Then $H$ is compact in $X$
by \cite[Theorem 2.5.2]{M1951}, hence it is metrizable since (2)
implies (5). Therefore, $(\mathcal{K}(H), \tau_{V})$ is compact and
metrizable, thus $\mathbf{H}\subset (\mathcal{K}(H), \tau_{V})$ is
metrizable.
\end{proof}

Since $\gamma$-space $X$ is a $D_0$-space, and every compact subset
of $X$ is metrizable, $(\mathcal{K}(X), \tau_V)$ is a $D_0$-space.
But, the following question still open.

\begin{question}
If $X$ is a $\gamma$-space, is $(\mathcal{K}(X), \tau_V)$ a
$\gamma$-space?
\end{question}

\section{The $G_\delta$-property of hyperspaces}
Finally we consider the characterization
of hyperspaces which have countable pseudocharacter.
 We say that a space $X$ has a {\it compact-$G_\delta$-property} if
every compact subset of $X$ is a $G_\delta$-set of $X$.

\begin{theorem}\label{t12}
Let $X$ be a space. Then $(\mathcal{K}(X), \tau_V)$ has a
compact-$G_\delta$-property if and only if $X$ has a
compact-$G_\delta$ property and every compact subset of $X$ is
metrizable.
\end{theorem}

\begin{proof}
Necessity. Assume that $(\mathcal{K}(X), \tau_V)$ has  the
compact-$G_\delta$-property. Obviously, $X$ has the
compact-$G_\delta$ property. Fix an any compact subset $K$ of $X$.
Then the hyperspace $\mathcal{K}(K)$ is compact by \cite{M1951}.
Since $\mathcal{K}(K)$ is regular and has compact-$G_\delta$
property, it follows that $\mathcal{K}(K)$ is a $D_0$-space, hence
$D_1$-space (note that every closed subset of $\mathcal{K}(K)$ is
compact since $\mathcal{K}(K)$ is compact). By Theorem ~\ref{t4},
$K$ is compact metrizable.

Sufficiency. Let $\mathbf{K}$ be a compact subset of
$\mathcal{K}(X)$, and let $H=\bigcup \mathbf{K}$. Then $H$ is
compact in $X$. By the assumption, $H$ is metrizable, hence
$\mathcal{K}(H)$ is compact metrizable. Let $\mathcal{B}=\{\langle
U_1(i), ..., U_{k_i}(i)\rangle: i\in \mathbb{N}\}$ be a countable
base of $\mathcal{K}(H)$, where each $k_{i}\in\mathbb{N}$, each
$U_j(i)$ ($i\in \mathbb{N}$ and $j\leq k_i$) is open in $H$. For any
$i\in \mathbb{N}$ and $j\leq k_i$, let $V_j(i)$ be an open subset of
$X$ such that $U_j(i)=V_j(i)\cap H$. Let $\{W_n: n\in \mathbb{N}\}$
be open subsets of $X$ with $H=\bigcap W_n$. For each
$n\in\mathbb{N}$, put
$$\mathcal{P}_n=\{\langle V_1(i)\cap W_n, ..., V_{k_i}(i)\cap
W_n\rangle: \langle V_1(i), ..., V_{k_i}(i)\rangle\cap
\mathbf{K}\neq \emptyset, i, n\in \mathbb{N}\};$$ we say that
$\langle V_1(i)\cap W_n, ..., V_{k_i}(i)\cap W_n\rangle$ is
associated with $\langle U_1(i), ..., U_{k_i}(i)\rangle$ for each
$i\in\mathbb{N}$. Put $$\mathcal{Q}_n=\{\bigcup \mathcal{P}':
\mathbf{K}\subset \bigcup \mathcal{P}', \mathcal{P}'\in
\mathcal{P}_n^{<\omega}\}$$ for each $n\in\mathbb{N}$; then let
$\mathcal{Q}=\bigcup_{n\in \mathbb{N}}\mathcal{Q}_n$. Clearly,
$|\mathcal{Q}|\leq \omega$. We claim that
$\mathbf{K}=\bigcap\mathcal{Q}$.

Indeed, it is trivial to verify that $\mathbf{K}\subset
\bigcap\mathcal{Q}$. Let $K$ be an any compact subset of $X$ with
$K\notin \mathbf{K}$. We prove that there is $\bigcup\mathcal{P}'\in
\mathcal{Q}$ such that $K\notin \bigcup\mathcal{P}'$.

\smallskip
{\bf Case 1:} $K\setminus H\neq \emptyset$.

\smallskip
Then there exists $j\in \mathbb{N}$ such that $K\setminus W_j\neq
\emptyset$, hence we can pick a $\bigcup\mathcal{P}'\in
\mathcal{Q}_j$, then $K\notin \cup\mathcal{P}'$.

\smallskip
{\bf Case 2:} $K\subset H$.

\smallskip Since $\mathcal{B}$ is a countable base of
$\mathcal{K}(H)$ and $K\notin \mathbf{K}$, it is easily checked that
there is a finite family $\mathcal{B}'\subset \mathcal{B}$ such that
$K\notin \bigcup\mathcal{B}'$, $\mathbf{K}\subset
\bigcup\mathcal{B}'$ and each element of $\mathcal{B}'$ meets
$\mathbf{K}$. Fix any $n\in \mathbb{N}$, and let $\mathcal{P}'$ be
the set of elements of $\mathcal{P}_n$ that are associated with the
elements of $\mathcal{B}'$. Then it is easily verified that $K\notin
\bigcup\mathcal{P}'$.
\end{proof}

\begin{remark}
From Theorems~\ref{tt44} and ~\ref{t12} we conclude that there
exists a space $X$ such that $(\mathcal{K}(X), \tau_V)$ has the
compact-$G_\delta$-property and is not a $D_{0}$-space, such as, the
Butterfly space \cite[Example 1.8.3]{LY2016}. Moreover, the
following questions are still unknown for us.
 \end{remark}

\begin{question}
If $X$ has the compact-$G_\delta$ property and every compact subset
of $X$ is metrizable, does $(CL(X), \tau_V)$ have a
compact-$G_\delta$-property?
\end{question}

\begin{question}
If $X$ has the compact-$G_\delta$ property and every compact subset
of $X$ is metrizable, does $(\mathcal{K}(X), \tau_F)$ or $(CL(X),
\tau_F)$ have a compact-$G_\delta$-property?
\end{question}

\begin{question}
Characterise spaces $X$ such that
$(\mathcal{K}(X), \tau_F)$, $(CL(X), \tau_F)$, $(CL(X), \tau_V)$ and
$(\mathcal{K}(X), \tau_V)$ are perfect, respectively?
\end{question}

In order to give a characterization of $(CL(X), \tau_F)$ with the countable pseudocharacter, we
need to prove an equality relating the pseudocharacter of hyperspace
$\psi (CL(X), \tau_F)$ to other cardinal functions. First, we recall
some concepts.

\begin{definition}\cite{E1989}
The {\it pseudocharacter of a point $x$} in a space $X$ is defined
as the smallest cardinal number of the form $|\mathcal{U}|$, where
$\mathcal{U}$ is a family of open subsets of $X$ such that
$\bigcap\mathcal{U}=\{x\}$; this cardinal number is denoted by
$\psi(x, X).$ The {\it pseudocharacter} of a space $X$ is defined as
the supremum of all numbers $\psi(x, X)$ for $x\in X$; this cardinal
number is denoted by $\psi(X)$. If $\psi(X)=\omega$, we say that $X$ has a {\it countable pseudocharacter}.
\end{definition}

\begin{definition}Let $X$ be a space. A collection of nonempty open sets $\mathcal{U}$
of $X$ is called a {\it $\pi$-base} if for every nonempty open set
$O$, there exists an $U\in\mathcal{U}$ such that $U\subset O$. The
{\it $\pi$-weight} of $X$ is defined as the infimum of all the
cardinal numbers of the $\pi$-bases of $X$; this cardinal number is
denoted by $\pi w(X)$. For a space $X$, define $$\pi
w_K(X)=\sup\{\pi w(A): A\ \mbox{is an arbitrary non-empty closed
subset of}\ X\}.$$
\end{definition}

\begin{definition}
Let $X$ be a space and $U$ be open in $X$. Denote
$$k_{U}=\inf\{|\mathcal{K}|: \mathcal{K}\ \mbox{is a family of
compact subsets of}\ X\ \mbox{such that}\ U=\bigcup\mathcal{K}\}$$
and $$k_{o}=\sup\{k_{U}: U\ \mbox{is open in}\ X\}.$$
\end{definition}

\begin{theorem}\label{t5}
For any space $X$, we have $\psi (CL(X), \tau_F)=\pi w_K(X)\cdot
k_o(X)$.
\end{theorem}

\begin{proof}
Suppose $\psi (CL(X), \tau_F)=\kappa$. We first prove that $\pi
w(A)\le \kappa$ for any $A\in CL(X)$, which implies that $\pi
w_K(X)\leq \kappa$. Fix any $A\in CL(X)$. Then $$A=\bigcap
\{W^-_{\alpha1}\cap\ldots\cap W^-_{\alpha
k_\alpha}\cap(K^c_\alpha)^+: \alpha<\kappa\},$$ where each
$W_{\alpha j}$ is open in $X$, $K_{\alpha}$ is compact in $X$ and
$k_\alpha\in \mathbb{N}$. Let $V_{\alpha j}=A\cap W_{\alpha j}$. Put
$$\mathcal{V}=\{V_{\alpha j}: j\leq k_\alpha,
\alpha<\kappa\}.$$ We claim that $\mathcal{V}$ is a $\pi$-base for
$A$. Indeed, let $U'$ be an any non-empty open subset of $A$; then
it can take an open subset $U$ in $X$ such that $U'=U\cap A$. Let
$H=A\setminus U$. Without loss of generality, we may assume $H\neq
\emptyset$, then $H\in CL(X)$. Hence there exists an $\alpha<\kappa$
such that $H\notin W^-_{\alpha 1}\cap...\cap W^-_{\alpha
k_\alpha}\cap(K^c_\alpha)^+$. Since $H\subset A\subset X\setminus
K_{\alpha}$, it follows that $H\in (K^c_{\alpha})^+$, which implies that $H\cap
W_{\alpha j}=\emptyset$ for some $j\in\mathbb{N}$. It follows that
$V_{\alpha j}\subset U'$. Hence $\pi w(A)\le \kappa$. Therefore,
$\pi w_K(X)\leq \kappa$.

Now we prove that $k_o(X)\leq\kappa$. Let $U$ be an any open subset
of $X$ such that $B=X\setminus U\in CL(X)$. From our assumption, it
follows that $$B=\bigcap \{W^-_{\alpha 1}\cap...\cap W^-_{\alpha
k_\alpha}\cap(K^c_\alpha)^+: \alpha<\kappa\},$$ where each
$W_{\alpha j}$ is open in $X$, $K_{\alpha}$ is compact in $X$ and
$k_\alpha\in \mathbb{N}$. We claim that
$U=\bigcup_{\alpha<\kappa}K_\alpha$. Indeed, for any $x\in U$, if
$x\notin \bigcup_{\alpha<\kappa}K_\alpha$, then $\{x\}\cup B\in
(K_\alpha)^c$ for each $\alpha<\kappa$, hence it follows that
$$\{x\}\cup B\in \bigcap \{W^-_{\alpha 1}\cap...\cap W^-_{\alpha
k_\alpha}\cap(K^c_\alpha)^+: \alpha<\kappa\},$$ which is a
contradiction. Therefore, $k_o(X)\leq\kappa$.

Conversely, suppose $\pi w_K(X)\leq \kappa$ and $k_o(X)\leq \kappa$.
Fix any $A\in CL(X)$. Let $\mathcal{V}'=\{V'_\alpha:
\alpha<\kappa\}$ be a $\pi$-base for $A$. For each $\alpha<\kappa$,
let $V_\alpha$ be an open set of $X$ with $V_\alpha\cap
A=V'_\alpha$. Put $\mathcal{V}=\{V_\alpha: \alpha<\kappa\}$, and let
$$\Delta(\mathcal{V})=\{\mathcal{W}: \mathcal{W}\subset\mathcal{V},
|\mathcal{W}|<\omega\}.$$ From our assumption, let $X\setminus
A=\bigcup_{\beta<\kappa}K_\beta$, where each $K_\beta$ is compact.
We claim that
$$\{A\}=\bigcap\{V^-_1\cap ...\cap V^-_k\cap
(K^c_\beta)^{+}: \{V_1, ..., V_k\}\in \Delta(\mathcal{V}),
\beta<\kappa\}.$$

Indeed, it is obvious that $A\in V^-_1\cap ...\cap V^-_k\cap
(K^c_\beta)^{+}$ for each $\{V_1, ..., V_k\}\in \Delta(\mathcal{V})$
and $\beta<\kappa$. Take any $$B\in \bigcap\{V^-_1\cap ...\cap
V^-_k\cap (K^c_\beta)^{+}: \{V_1, ..., V_k\}\in \Delta(\mathcal{V}),
\beta<\kappa\}.$$ If $B\setminus A\neq \emptyset$, then pick any
$x\in B$, hence $x\in K_\beta$ for some $\beta<\kappa$; however,
$B\subset K^c_\beta$ for any $\beta<\kappa$, which is a
contradiction. If $A\setminus B\neq \emptyset$, then there is
$V'_i\in \mathcal{V}'$ such that $V'_i\subset A\setminus B$, hence
there exists $V_i\in \mathcal{V}$ such that $V'_i=V_i\cap A$, and
$V_i\cap B=\emptyset$, which is a contradiction. Therefore, $A=B$.
\end{proof}

\begin{theorem}\label{t55}
If $(CL(X), \tau_F)$ is Hausdorff and has countable pseudocharacter,
then $(CL(X), \tau_F)$ is first-countable.
\end{theorem}

\begin{proof}
By Theorem ~\ref{t5}, every closed subset of $X$ is separable and
$X$ is hemicompact. We prove that $X$ is first-countable. Since
$(CL(X), \tau_F)$ is Hausdorff, it follows from \cite[Proposition
5.1.2]{B1993} that $X$ is locally compact. Then $X$ is first-countable since a Hausdorff,
locally compact space with a countable pseudocharacter is
first-countable. By \cite[Corollary 7(1)]{B1993}, $(CL(X), \tau_F)$
is first-countable.
\end{proof}

\begin{corollary}\label{ccccc}
If $(CL(X), \tau_F)$ is Hausdorff and has compact-$G_\delta$
property, then $(CL(X), \tau_F)$ is a $D_0$-space.
\end{corollary}

\begin{proof}
Since $(CL(X), \tau_F)$ is Hausdorff. Then $X$ is locally compact
\cite[Proposition 5.1.2]{B1993}. By Theorem ~\ref{t5}, every closed
subset of $X$ is separable and $X$ is hemicompact, hence $X$ is
paracompact. It is easy to see that each compact subset is
metrizable by Lemma ~\ref{l4}, then $X$ is metrizable. Therefore,
$(CL(X), \tau_F)$ is a $D_0$-space by Theorem ~\ref{t10}.
\end{proof}

\begin{remark}
There exists a hyperspace $(CL(X), \tau_F)$ of a space $X$ such that $(CL(X), \tau_F)$ has a countable pseudocharacter, but $(CL(X), \tau_F)$ does not have compact-$G_{\delta}$-property. Indeed, let $X$ be the
 Alexandroff double-arrow space, see \cite[Example 1.8.9]{LY2016}. It is well known that $X$ is Hausdorff, first-countable and compact. Then it follows from \cite[Example 8]{Ho1998} that $(CL(X), \tau_F)$ is first-countable, hence it has a countable pseudocharacter; however, from Theorem~\ref{t12}, it follows that $(CL(X), \tau_F)$ does not have
compact-$G_\delta$-property because $X$ is compact and non-metrizable.
\end{remark}

{\bf Acknowledgements}. The authors wish to thank
the reviewers for careful reading preliminary version of this paper and providing many valuable suggestions.

\end{document}